\documentclass[preprint,12pt]{amsart}

\usepackage[margin=1.5in]{geometry}

\usepackage{amsmath}
\usepackage{amsfonts}
\usepackage{amssymb}
\usepackage{xypic}
\usepackage{longtable}
\usepackage{amsthm}

\theoremstyle{plain}
\newtheorem{theorem}{Theorem}[section]
\newtheorem{corollary}[theorem]{Corollary}
\newtheorem{lemma}[theorem]{Lemma}
\newtheorem{proposition}[theorem]{Proposition}
\theoremstyle{definition}

\newtheorem{rem}[theorem]{Remark}
\newtheorem{con}[theorem]{Conjecture}
\newtheorem{que}[theorem]{Question}

\DeclareMathOperator{\Km}{Km}

\DeclareMathOperator{\Span}{Span}

\makeatletter
\def\ps@pprintTitle{%
  \let\@oddhead\@empty
  \let\@evenhead\@empty
  \let\@oddfoot\@empty
  \let\@evenfoot\@oddfoot
}
\makeatother

\usepackage[usenames,dvipsnames]{color}

\usepackage[backref=page]{hyperref}

\hypersetup{
 colorlinks,
 citecolor=green,
 linkcolor=blue,
 urlcolor=Blue}

\newcommand{\G}{\mathcal{G}}

\title{The Gauss map and secants of the Kummer variety}
\thanks{The first author was partially supported by Fondecyt Grant 3150171 and CONICYT PIA ACT1415}
\thanks{The second author is supported by prize funds related to the PRIN-project 2015EYPTSB ``Geometry of Algebraic Varieties" and by University Roma Tre and the FIRB 2012 ``Moduli spaces and their applications" .}
\thanks{The third author was partially supported by Progetto di Ateneo  moduli, deformazioni e superfici K3 and  PRIN 2015  : Moduli spaces and Lie theory}

\author{Robert Auffarth}

\address{Departamento de Matem\'aticas, Facultad de
Ciencias, Universidad de Chile, Santiago\\Chile}
\email{rfauffar@uchile.cl}

\author{Giulio Codogni}

\address{Dipartimento di Matematica e Fisica, Universit\`a Roma Tre, Italy}
\email{codogni@mat.uniroma3.it}

\author{Riccardo Salvati Manni}

\address{Dipartimento di Matematica ``Guido Castelnuovo", Universit\`a di Roma ``La Sapienza", Italy}
\email{salvati@mat.uniroma1.it}

\begin{document}

\maketitle

\begin{abstract}
Fay's trisecant formula shows that the Kummer variety of the Jacobian of a smooth projective curve has a four dimensional family of trisecant lines. We study when these lines intersect the theta divisor of the Jacobian, and prove that the Gauss map of the theta divisor is constant on these points of intersection, when defined. We investigate the relation between the Gauss map and multisecant planes of the Kummer variety as well.
\end{abstract}

\begin{section}{Introduction}

Let $(A,\Theta)$ be a complex indecomposable principally polarised abelian variety of dimension $g$. The line bundle $2\Theta$ is canonically defined, and gives a finite morphism
$$
\Km\colon A\to \mathbb{P}^{2^g-1}
$$
whose image is the Kummer variety $K(A)$ of $A$.

When $A$ is the Jacobian $J(C)$ of a curve $C$, the Kummer variety has a four dimensional family of trisecants, and this is indeed a characterization of Jacobians, see Sections \ref{branch}, \ref{ram} and references therein.

In this paper, we fix a symmetric theta divisor representing the polarization, which by abuse of notation we still denote by $\Theta$, and we look for lines that intersect $\Km(\Theta)$ in at least three points; we will call these lines \textit{theta trisecants} or trisecants of the theta divisor. Let us remark that any two symmetric theta divisors are related by a translation of order two, and these translations are the only ones induced by projective automorphisms of $\mathbb{P}^{2^g-1}$, so the study of trisecants of a symmetric theta divisor is independent on the choice of the divisor. We do not study trisecants of non-symmetric theta divisors.

The main novelty of this work is to relate these trisecants to the Gauss map. This is used to show that these trisecants exist, and to completely classify them.

Recall that the Gauss map is a dominant rational morphism
$$
\G\colon \Theta \dashrightarrow \mathbb{P}T_0A^*\cong \mathbb{P}^{g-1}
$$
whose domain is the smooth locus $\Theta^{sm}$. A basic reference is \cite[Section 4.4]{BL}, some recent research papers on the Gauss map and their generalizations are \cite{Giulio}, \cite{Riccardo}, \cite{Kramer} and \cite{Pareschi}.

In the case of the Jacobian of a smooth projective genus $g$ curve $C$, the Gauss map has an explicit geometric interpretation. Indeed, $\mathbb{P}T_0A^*$ is canonically isomorphic to the canonical linear system $\mathbb{P}H^0(C,K)$, the divisor $\Theta$ can be seen as the locus of effective divisors $W_{g-1}$ in $\mbox{Pic}^{g-1}(C)$, and if $D\in W_{g-1}$ is a smooth point, $\G(D)$ is just the linear span of $\varphi(D)$ in $\mathbb{P}^{g-1}$, where $\varphi:C\to\mathbb{P}^{g-1}$ is the canonical map. Note that each smooth point of $\Theta$ thus defines a canonical divisor $\varphi^{-1}(\mathcal{G}(D))$. 

Our main result is the following.

\begin{theorem}[= Theorems \ref{thm:exist} and \ref{thm:proof}]\label{thm:intro}
Let $A$ be the Jacobian of a smooth projective curve, and let $x,y\in\Theta^{sm}$. If the images of $x$ and $y$ in the Kummer variety lie on a trisecant line, then $\G(x)=\G(y)$, $\G(x)$ lies in a  two dimensional subvariety $\mathcal{B}^*_3$ of the branch locus of $\G$ defined in Section \ref{branch}, and $x$ and $y$ are among the points of highest multiplicity on the fiber of $\G$. 

Conversely, for a generic $K_0\in\mathcal{B}^*_3$, there exists a trisecant of the theta divisor whose points of intersection with $\Theta^{sm}$ lie in the fiber of $\G$ over $K_0$, and the points of highest multiplicity of $\G$ over $K_0$ all lie on trisecants.
\end{theorem}

The relation between trisecants and the Gauss map was in some sense hidden in the literature about the Torelli Theorem. Indeed, on one hand, the classical proof of the Torelli Theorem given by Andreotti in \cite{A} relies on the study of the branch locus of the Gauss map; on the other hand, it is possible to prove the Torelli Theorem as a consequence of the analysis of trisecants of a Jacobian, cf. \cite{G} and \cite[Page 267]{ACGH}.

In Section \ref{multi}, we generalize part of our results to multisecants of the theta divisor. Recall that multisecants do not characterize Jacobians, so we do not know if this generalization can be pushed any further.

In Section \ref{Gamma}, we investigate the linear span in $\mathbb{P}^{2^g-1}$ of fibers of the Gauss map, and we point out a relation with the $\Gamma_{00}$ Conjecture.

It is now well-known (see \cite{BeauvilleDebarre,GK}) that if $(A,\Theta)$ is a general Prym variety, then its Kummer variety possesses a family of quadrisecant planes. We do not know what relation these planes have with respect to the Gauss map, and this is certainly an interesting direction for research. By using Proposition \ref{p1} below, we can only prove that if we have four points of the theta divisor of a Prym that lie on a quadrisecant plane on the Kummer variety, then the linear span of the images of these points via the Gauss map is at most one-dimensional.

\medskip

\noindent\textit{Acknowledgements:} The first author would like to thank Sapienza Universit\`a di Roma as well as the Universit\`a Roma Tre where this work was started.

\end{section}

\begin{section}{Preliminaries about multisecants of the theta divisor}

A multisecant linear space is a $k$ dimensional linear subspace of $\mathbb{P}H^0(A,2\Theta)^*$ which intersects the Kummer variety in at least $k+2$ points. We say that this is a multisecant of the theta divisor if these $k+2$ points are on the image of the theta divisor. In this section, we investigate the behavior of the Gauss map on these intersection points.

\begin{proposition}\label{p1}
Let $A$ be an abelian variety and let $\Theta$ be a symmetric divisor that induces a principal polarization on $A$. Let $x_1\in A$ and $x_2,\ldots,x_r\in\Theta $ be points such that the points
$$\Km(x_1),\dots,\Km(x_r)$$ 
are contained in an $r-2$ dimensional linear subspace of  $\mathbb{P}H^0(A,2\Theta)^*$, but any subset  of order $r-1$ is in general position. Then $x_1$ lies in $\Theta$, and $\mathcal{G}(x_1),\ldots,\mathcal{G}(x_r)$, if defined,  are contained in an $(\lfloor r/2\rfloor-1)$ linear subspace of $\mathbb{P}^{g-1}$. Moreover, if $r-1$ points lie on the singular locus of $\Theta$, the $r$-th one also does. 

\end{proposition}
\begin{proof}
Using theta functions with characteristics, the fact that the points $\Km(x_i)$ are not in general position can be translated into the following linear relations:
$$\Theta[\epsilon](\tau,x_1)=\sum_{k=2}^{r}\alpha_k\Theta[\epsilon](\tau,x_k)$$
for all $\epsilon\in(\mathbb{Z}/2\mathbb{Z})^g$. Let $s_x(z)$ be the section of $2\Theta$ defined as $\theta(z-x)\theta(z+x)$, where $\theta$ is a non-zero section of $H^0(A,\Theta)$. Since $\Theta$ is symmetric we have that $\theta$ is even. By applying the Addition Formula for theta functions, we obtain that there exist $\beta_2,\ldots,\beta_r\in\mathbb{C}$ such that
$$s_{x_1}=\sum_{k=2}^r\beta_k s_{x_k}.$$
This already shows that $x_1$ is in $\Theta$, and proves the statement about singularities.

Denote by $\nabla_x\theta$ the gradient of the theta function evaluated at $z=x$. Since in coordinates $\G(x_i)=\nabla_{x_i}\theta$, to prove the proposition we have to show that the $\nabla_{x_1}\theta$ span a $(\lfloor r/2\rfloor-1)$ dimensional linear space.

Now for $x\in\Theta$, we have $\frac{\partial^2s_x}{\partial z_i\partial z_j}(0)=2\frac{\partial\theta}{\partial z_i}(x)\frac{\partial\theta}{\partial z_j}(x)$, hence
$$\frac{\partial\theta}{\partial z_i}(x_1)\frac{\partial\theta}{\partial z_i}(x_1)=\sum_{k=2}^r\beta_k\frac{\partial\theta}{\partial z_i}(x_k)\frac{\partial\theta}{\partial z_i}(x_k).$$
Therefore
$$(\nabla_{x_1}\theta)(\nabla_{x_1}\theta)^t=\sum_{k=2}^r\beta_k(\nabla_{x_k}\theta)(\nabla_{x_k}\theta)^t\,,$$
and we can assume that $\beta_k\neq 0$ for all $k$, since any subset of order $r-1$ of $\Km(x_1),\ldots,\Km(x_r)$ is in general position.

We conclude thanks to the following lemma.
\begin{lemma}[Lemma 1 \cite{Ig}] Let  $V$ be a vector space over a field $K$, take $r$ vectors $v_1, \dots, v_r$ such that
$$\sum_{i=1}^r a_i v_i\otimes  v_i=0$$ 
for some non-zero $a_1,\dots, a_r\in K$. Then, the vectors $v_1, \dots, v_r$ span a space of dimension at most $\lfloor r/2\rfloor$.
\end{lemma}

\end{proof}

Let us spell out our result for trisecants, so $r=3$ and $k=1$.

\begin{corollary}\label{c1} 
Let us assume that points the $x_1,x_2,x_3 \in K(A)$ lie on a trisecant.
\begin{enumerate}
\item If two of them, say  $x_1,x_2,$ are in $\Theta$, then also $x_3$ is in $\Theta$.

\item If the three  points are in  the theta divisor  and two of them, say  $x_1,x_2,$ are singular, then also $x_3$ is singular.

\item If the three  points are in  the theta divisor, then $\mathcal{G}(x_i) $ is constant for all $x_i \in{\Theta} ^{sm} $.
\end{enumerate}
\end{corollary}
  
\end{section}

\begin{section}{Branch locus of the Gauss map and construction of trisecants}\label{branch}

Let $A=J(C)$ be the Jacobian of a smooth projective curve. We fix a symmetric theta divisor $\Theta$, and let $\kappa$ be the associated theta characteristic, which we will also use as a divisor on the curve. Denote by $\mathbb{P}^{g-1}$ the canonical system $|K|=H^0(C,K)$, which is canonically isomorphic to $\mathbb{P}( T_0J(C))^*$.

Let $\varphi:C\to\mathbb{P}^{g-1}$ be the canonical map. We introduce a stratification $\mathcal{B}_\ell$ of the branch locus of the Gauss map $\mathcal{G}$. For $1\leq\ell\leq g-1$ let
$$\mathcal{B}_\ell:=\left\{H\in(\mathbb{P}^{g-1})^*:H\cap \varphi(C)\mbox{ is of the form }\sum_{i=1}^{2\ell-2}P_i+2\sum_{i=1}^{g-\ell}Q_i\mbox{ for }P_i,Q_j\in \varphi(C)\right\}.$$
Notice that
$$\mathcal{B}_{1}\subseteq\mathcal{B}_{2}\subseteq\cdots\subseteq\mathcal{B}_{g-1}$$
and $\mathcal{B}_1$ is the set of effective theta characteristics.
\begin{rem}
The locus $\mathcal{B}_{g-1}$ is the branch locus of the Gauss map. The dual variety of this locus is the canonical curve, and this is indeed Andreotti's proof of the Torelli Theorem, cf. \cite{A}. The dual of the loci $\mathcal{B}_{\ell}$ should be related, at least for large values of $\ell$, to multisecant varieties of the canonical curve and special linear series. We do not have a nice description of these loci.
\end{rem}

We start off computing the dimension of the loci $\mathcal{B}_{\ell}$.

\begin{lemma}
\label{lem:exis}
For every $1\leq \ell \leq g-1$, the space $\mathcal{B}_{\ell}$ is of dimension at least $\ell-1$. Moreover, the union $\mathcal{B}_{\ell}^*$ of irreducible components of dimension exactly $\ell-1$ such that the generic geometric points represent divisors where the points $P_i$ are all distinct is non-empty.
\end{lemma}
\begin{proof}

Let $C_k$ be the symmetric product of $k$ copies of the curve $C$. The linear system $|K|=(\mathbb{P}^{g-1})^*$ can be embedded in $C_{2g-2}$, and the space $\mathcal{B}_{\ell}$ is its intersection with the image of the map
 $$C_{g-\ell}\times C_{2\ell-2}\to  C_{2g-2}$$
 $$(D, E) \mapsto 2D+E$$
This intersection, if non-empty, is at least $\ell-1$ dimensional, and $\mathcal{B}_{\ell-1}$ is of codimenison at most one in $\mathcal{B}_{\ell}$.

The locus $\mathcal{B}_1$ is the locus of effective theta characteristics, and it is known to be non-empty and zero dimensional. The loci $\mathcal{B}_{\ell}$ contain $\mathcal{B}_1$, hence they are non-empty. Locally around $\mathcal{B}_1$, the dimension of $\mathcal{B}_{\ell}$ must be exactly $\ell-1$, otherwise $\mathcal{B}_1$ would have dimension strictly greater than zero, hence $\mathcal{B}_{\ell}^*$ is non-empty.

\end{proof}

The locus $\mathcal{B}_{g-1}$, being the branch locus of the Gauss map, is of codimension one in $(\mathbb{P}^{g-1})^*$, so it is equal to $\mathcal{B}_{g-1}^*$. The generic tangent hyperplane to the canonical curve is tangent to a single point, cf \cite[Corollary 2.4]{Zak}, hence $\mathcal{B}_{g-2}$ is of dimension $g-3$ and it is equal to $\mathcal{B}_{g-2}^*$. The locus $\mathcal{B}_1$, being the locus of theta characteristics, is zero dimensional and equal to $\mathcal{B}_1^*$. We do not know about the other $\mathcal{B}_{\ell}$.

We are now going to use the locus $\mathcal{B}_3^*$ to construct trisecants of the theta divisor. Recall that, thanks to Fay's trisecant formula, we can construct trisecants out of four points on the curve and a consistent way to divide by two on the Jacobian; more explicitly we have the following theorem, cf. \cite{F} or \cite[Section IIIb]{M}.
\begin{theorem}[Fay's trisecant formula]\label{thm:Faytrisecant}
Let $p$, $q$, $r$ and $s$ be points on $C$, and $a$, $b$ and $c$ points in $A=J(C)$ such that $a\in 2^{-1}\mathcal{O}_C(p-q-r+s)$, $b\in 2^{-1}\mathcal{O}_C(p-q+r-s)$ and $c\in 2^{-1}\mathcal{O}_C(p+q-r-s)$, with $a+b=\mathcal{O}_C(p-q)$ and $a+c=\mathcal{O}_C(p-r)$. Then the images of $a$, $b$ and $c$ on the Kummer variety lie on a trisecant.
\end{theorem}

We are now ready to construct trisecants of the theta divisor.

\begin{theorem}[Existence of trisecants of the theta divisor]\label{thm:exist}
Let $K_0=p+q+r+s+2D$ be a point of $\mathcal{B}_3$ as in Lemma \ref{lem:exis}. Let $\kappa$ be the theta characteristic associated to the theta divisor $\Theta$, and let 
 $$a=\mathcal{O}_C(p+s+D-\kappa) \,  , \,\, b =\mathcal{O}_C(p+r+D - \kappa) \, ,\, c=\mathcal{O}_C(p+q+D-\kappa)\,.$$
 Then $a,b,c\in\Theta$, their images in the Kummer variety lie on a trisecant, and the Gauss map evaluated at any of these points (if defined) equals $K_0$. Furthermore, if $K_0$ is generic in $\mathcal{B}_3^*$, then these three points are distinct.
\end{theorem}
\begin{proof}
To show that $a$ lies on the theta divisor, we have to prove that $a+\kappa$ is effective, and this follows from the definition; similarly, also $b$ and $c$ lie on the theta divisor.

As explained in Theorem \ref{thm:Faytrisecant}, to prove that they are collinear, we have to show that $2a$ is linearly equivalent to $p-q-r+s$, $2b$ to $p-q+r-s$, $2c$ to $p+q-r-s$, $a+b$ to $p-q$ and $a+c$ to $p-r$. This is fine because $2D-2\kappa$ is linearly equivalent to $-(p+q+r+s)$.

As shown in Lemma \ref{lem:exis}, for a generic geometric point $K_0$ of $\mathcal{B}^*_3$ the points $p$, $q$, $r$ and $s$ are distinct, hence also $a$, $b$ and $c$ are distinct.

\end{proof}

If the point $K_0$ is chosen generically in $\mathcal{B}^*_2$, we obtain a degenerate trisecant, so a line which is tangent to the theta divisor and it furthermore intersects it at another distinct point. We do not know if the theta divisor has a tangent of order three. 

For genus three non-hyperelliptic curves, the theta divisor is smooth and the Gauss map is finite of degree 6. Generically over the locus $\mathcal{B}_3^*$ we have the points $a,b,c$ and $-a,-b,-c$.

Let us discuss the example of genus four curves, we refer to \cite[Page 232]{ACGH} for the basic facts. We take a generic curve, so that the theta divisor of the Jacobian has just two singular points, and they are not of order two. For dimensional reasons, the generic trisecant provided by Theorem \ref{thm:exist} intersects the theta divisor in three distinct smooth points. On the other hand, we can take a canonical divisor $K_0=p+q+r+s+2Q$ such that the points are all distinct and $p+q+Q$ is a $\mathfrak{g}^1_3$; to see that such a divisor exists, start off from a generic $\mathfrak{g}^1_3$, say $p+q+Q$, in the canonical model these three points lie on a line, take now a plane $\Pi$ containing this line and tangent to the canonical curve at $Q$, then such a plane cuts out the requested canonical divisor. For this divisor, the associated trisecant intersects the theta divisor at two smooth points and one singular point. 

We do not know if there exists a trisecant intersecting the theta divisor at three distinct singular points.

\end{section}

\begin{section}{Ramification of the Gauss map and trisecants}\label{ram}
 
This section is devoted to the proof of the following theorem.
\begin{theorem}\label{thm:proof}
Let $a$, $b$ and $c$ be three points on the theta divisor whose images in $\mathbb{P}^{2^g-1}$ lie on a trisecant. Moreover, assume that $a$ and $b$ are smooth points of the theta divisor. Then
$$
\G(a)=\G(b)=K_0=p+q+r+s+2(P_1+\cdots P_{g-3})
$$
In particular, this is a point of $\mathcal{B}_{3}$. Morever, the points $a$ and $b$ have the highest multiplicity in the fibre of $\G$ over $K_0$; if the points $P_i$ are distinct, this multiplicity is $2^{g-3}$.

\end{theorem}

\begin{corollary} If $a=b$, i.e the secant is a tangent  then   $\G(a)=2D+p+2q +s$  Similarly If the three points  coincide, then $\G(a)=2D+p+3q $ 
\end{corollary}

Recall that the equality $\G(a)=\G(b)$ has been proved in Corollary \ref{c1}. To prove our result, we also need to know that all trisecants are obtained out of Fay's formula, see Theorem \ref{thm:Faytrisecant}; this is the trisecant conjecture, which we now recall, cf. \cite{G}, \cite{ad}, \cite{sh}, \cite{W} and \cite{kr}.

\begin{theorem}[Trisecant conjecture]\label{thm:trisecant_conj}
Let $a$, $b$ and $c$ be three points on the Kummer variety $K(A)$ which lie on a trisecant. Then the abelian variety $A$ is a Jacobian of a curve $C$, and there exist four points $p$,$q$, $r$ and $s$ on $C$, such that $a\in 2^{-1}\mathcal{O}_C(p-q-r+s)$, $b\in 2^{-1}\mathcal{O}_C(p-q+r-s)$ and $c\in 2^{-1}\mathcal{O}_C(p+q-r-s)$, with $a+b=\mathcal{O}_C(p-q)$ and $a+c=\mathcal{O}_C(p-r)$.
\end{theorem}

Let us fix some notations. Given two divisors $A$ and $B$ on the curve, we write $A\equiv B$ if they are linearly equivalent, $A=B$ if they are equal as divisors, and we will write $l(A):=\dim H^0(C,\mathcal{O}_C(A))$. To start with, let us prove the following preliminary lemmas:
\begin{lemma} \label{effe}
For any  effective $D$ with $l(D)=1$ and for every $P\in C$ 
$$ l(D)>l(D-P)  \iff  P \notin  {\rm Supp} (D)$$    
\end{lemma}

\begin{proof}
$|D|=|D-P|$ if and only if  effective divisors coincide, if and only if $ P \in  {\rm Supp} (D)$.
\end{proof}
\begin{lemma}\label{lem:trick}
Let $K_0$ be a canonical divisor. Suppose we can write
$$K_0=A_1+B_1=A_2+B_2\,,$$
where $A_i$ and $B_i$ are effective, of degree $g-1$, and are not special (a divisor $D$ is not special if $l(K_0-D)=0$).  Let $p_1$ and $p_2$ be two points on the curve such that
$$p_1-p_2\equiv A_1-A_2$$
then $p_1$ is in the support of $A_1$ and $B_2$, and $p_2$ is in the support of $A_2$ and $B_1$.

\end{lemma}
\begin{proof}
 Obviously  we  have also   $p_2-p_1\equiv B_1-B_2$. We first show that $p_1$ is in the support of $A_1$ if and only if $p_2$ is in the support of $A_2$. Since
$$A_1-p_1=K_0-B_1-p_1 \equiv K_0-B_2-p_2= A_2-p_2$$
applying Lemma \ref{effe}, we have the conclusion.

We now show that $p_1$ is in the support of $A_1$. Arguing by contradiction, thanks to Riemann Roch, we have that $l(B_1+p_1)=l(B_2+p_2)=1$. This shows that
 $B_1+p_1=B_2+p_2$, hence we can write
$$ B_1= P_1+\cdots + P_{g-2}+p_2 \quad \textrm{and} \quad B_2=  P_1+\cdots  + P_{g-2}+p_1\,.$$ 
Thus  $p_1$ and $p_2$ are in the support of $K_0$, hence $p_1\in B_1$ and e $p_2\in B_2 $
This  gives  $B_1=B_2$ that is a contradiction.The other statement is obtained by symmetry.
\end{proof}

We can now prove the central part of our claim.

\begin{proposition}\label{centralprop}
Let $a,b\in \Theta^{sm}$ and $c\in\Theta$ as in Theorem \ref{thm:trisecant_conj}, then
$$
\mathcal{G}(a)=p+q+r+s+2(P_1+\cdots P_{g-3})\,,
$$
where $P_i$ are points on the curve. 
\end{proposition}
\begin{proof}

The assumption $a,b\in \Theta^{sm}$ means that
$$a=D_1-\kappa,\,b=D_2-\kappa $$
with $D_i$ effective and $l( D_i)=1$, and $\kappa$ the theta characteristic associated to $\Theta$.

Because of Corollary \ref{c1}, the points $a,b$  have the same image via  Gauss map; this means that the divisors $D_i$  determine a unique  canonical  divisor $K_0$ with 
$$ K_0= D_i+E_i\, ,$$ 
$D_i, E_i$ effective and $l(D_i)=l(E_i)=1$.

Look now at the difference $a-b=s-r\equiv D_1-D_2\equiv E_2-E_1$. Remember we are using the same notation as Theorem \ref{thm:trisecant_conj}. Lemma \ref{lem:trick} shows that $r$ is in the support of $E_1$ and $D_2$, and $s$ is in the support of $E_2$ and $D_1$.
 Applying the Lemma \ref{lem:trick} also to     $a+b=p-q \equiv D_1-E_2\equiv D_2-E_1$   we can show that $r, q$ are in the support of $E_1$, $s, q$ are in the support of $E_2$,  $p,s$ is in the support of $D_1$, $p,r$ is in the support of $D_2$  

Now look again at the difference  $a-b= s-r\equiv D_1-D_2$, and write it as $D_1-s\equiv D_2-r$. Since $s$ in the support of $D_1$ and $r$ is the support of $D_2$, Lemma \ref{effe} implies that the equality is an equality of divisors, so we can write
 $$D_1=  P_1+\dots+ P_{g-3}+p+s, \, D_2=P_1+\dots+ P_{g-3}+p+r \,.$$
 Keep on playing this trick we get that
$$E_1=  Q_1+\dots+ Q_{g-3}+q+r, \, E_2=Q_1+\dots+ Q_{g-3}+q+s \,.$$

 Using $a+b= p-q \equiv D_1-E_2$ we get that
 $D_1-p=E_2-q$, hence $$P_i=Q_i, \, i=1,\dots (g-3).$$
In particular
 $$ K_0=2( P_1+\dots+ P_{g-3})+p+r+q+s$$
as required.

\end{proof}

The last part of Theorem \ref{thm:proof} follows from the next proposition.

\begin{proposition}\label{prop:count_mult}
Let $K_0= n_1 P_1+\dots +n_kP_{k}$ be a canonical divisor, and let $D= l_1 P_1+\dots +l_kP_{k}$ be a degree $g-1$   divisor such that  l(D)=1  and $K_0>D$; then the multiplicity of $D-\kappa$ in the fibre $\G^{-1}(K_0)$ is
$$m=\binom{n_1}{l_1} \cdots  \binom{n_k}{l_k}$$  

\end{proposition}
\begin{proof} The divisor $D-\kappa$ represents a smooth point of the theta divisor with $\G(D)=K_0$. 

On a Jacobian, all fibers of the Gauss map are finite; the domain of the Gauss map is the smooth locus of $\Theta$, which, being by construction smooth, is Cohen-Macauly; we conclude that the Gauss map is flat (see for instance \cite[Exercise III.10.9]{Har}).

Let $\Delta$ be the spectrum of a DVR (or, if one prefers to work in the analytic category, a small disc), and take an embedding $\iota\colon \Delta \hookrightarrow \mathbb{P}^{g-1}$ such that the closed point $0$ maps to $K_0$, and the generic point $\eta$ maps to a reduced divisor $Q_1+\cdots +Q_{2g-2}$. We label the divisors $Q_i$ so that the divisor $\sum_{i=1}^{n_1}Q_i$ specializes to $n_1P_1$, the divisor $\sum_{i=n_1+1}^{n_2}Q_i$ specializes to $n_2P_2$, and so on.

Let $X$ be the irreducible component of $\mathcal{G}^{-1}(\iota(\Delta))$ containing $D-\kappa$. The base change $\mathcal{G}\colon X\to \Delta$ of the Gauss map is again flat, because flatness is preserved by base change. The fibre over $0$ is supported on $D-\kappa$, whereas the fibre over $\eta$ consists of reduced points. To compute the requested multiplicity it is, by flatness, enough to compute the number of points in the fibre over $\eta$. Every element in this fibre specializes to $D-\kappa$, hence to determine one of them we have to choose $l_1$ points in $\{Q_1,\dots,Q_{n_1}\}$, $l_2$ points in $\{Q_{n_1+1},\dots , Q_{n_1+n_2}\}$, and so on. 
\end{proof}

\end{section}

\begin{section}{Ramification of the Gauss map and multisecants}\label{multi}

In this section, we show that the stratification of branch locus of the Gauss map introduced in  Section \ref{branch} can be used to construct multisecants. To this end, we need a generalization of Fay's trisecant formula due to Gunning, see \cite{G2} and \cite[Section 7]{Grus}.

\begin{theorem}[Gunning multisecant formula]\label{Gunn}
Let $C$ be a smooth projective curve, let $p_1,\ldots,p_\ell,q_1,\ldots,q_{\ell-2}\in C$ be different points, and take line bundles
$$a_j\in\frac{1}{2}\mathcal{O}_C\left(2p_j+\sum_{i=1}^{\ell-2}q_i-\sum_{i=1}^\ell p_i\right)$$
for $j=1,\ldots,\ell$, such that $a_j+a_k=\mathcal{O}_C(p_j+p_k+\sum_{i=1}^{\ell-2}q_i-\sum_{i=1}^{\ell}p_i)$ for all $j,k$. Then the images of $a_1,\ldots,a_\ell$ in the Kummer variety lie on an $(l-2)$-plane. 
\end{theorem}

In particular, this gives a $2\ell-2$ dimensional family of linear subvarieties of dimension $\ell-2$ that intersect the Kummer variety in at least $\ell$ points. Note that the condition on $a_j+a_k$ assures us that we are dividing by $2$ in a uniform way; in other words, this is equivalent to dividing by two on the universal cover of the Jacobian. 

We are now going to use the loci $\mathcal{B}_{\ell}$ defined in Section \ref{branch} to construct $\ell-2$ dimensional multisecants of the theta divisor.

\begin{theorem}[Existence of multisecants of the theta divisor]\label{thm:multi} Let 
$$K_0=\sum_{i=1}^{2\ell-2}P_i+2\sum_{j=1}^{g-\ell}Q_j$$
be a generic point of $\mathcal{B}_{\ell}^*$, so that the points $P_i$ are distinct, and let
$$\Sigma=\{p_1,\ldots,p_\ell\}\cup\{q_1,\ldots,q_{\ell-2}\}$$ 
be a partition of $\{P_1,\ldots,P_{2\ell-2}\}$. Define
$$D_j^\Sigma:=p_j+\sum_{i=1}^{\ell-2}q_i+\sum_{j=1}^{g-\ell}Q_j\hspace{1cm}\mbox{and}\hspace{1cm}a_j^\Sigma:=D_j^{\Sigma}-\kappa$$
for $j=1,\ldots,\ell$. Then $a_j^\Sigma\in\Theta$ for all $j$ and the images of $a_1^\Sigma,\ldots,a_\ell^\Sigma$ in the Kummer variety lie on an $(\ell-2)$-plane. 

Moreover, the Gauss map is constant on $\mathcal{S}:=\bigcup_{\Sigma}\{a_1^\Sigma,\ldots,a_\ell^\Sigma\}\cap\Theta^{sm}$, and $\mathcal{S}$ consists of precisely the elements of the fiber of the Gauss map of highest multiplicity.
\end{theorem}
\begin{proof}
It is clear that $a_j^\Sigma\in\Theta$ for all $j$ by construction, and by Theorem \ref{Gunn} the images of the $a_j^\Sigma$ lie on an $\ell-2$ dimensional linear variety. 

Now take $a_j^\Sigma\in \mathcal{S}$. Then $D_j^\Sigma$ spans a unique hyperplane in $\mathbb{P}^{g-1}$. Note moreover that $D_j^\Sigma\leq K_0$ and so $\mbox{span}(K_0)$ is this hyperplane, which is independent of $a_j^\Sigma$. 

Note that in $\bigcup_{\Sigma}\{a_1^\Sigma,\ldots,a_\ell^\Sigma\}$ there are precisely $\binom{2\ell-2}{\ell-1}$ elements. Moreover, thanks to Proposition \ref{prop:count_mult}, the highest multiplicity over $K_0$ is $2^{g-\ell}$ and there are at most $\binom{2\ell-2}{\ell-1}$ points that give this multiplicity. The highest multiplicity of a point on the fiber over an element of $\mathcal{B}_\ell$ is $2^{g-\ell}$, and the points that have this multiplicity are of the form
$$\sum_{i=1}^{g-\ell}Q_i+\sum_{i=1}^{\ell-1}P_{j_i}$$
for certain $j_i\leq 2\ell-2$, and these are exactly the multisecant points we are looking at.
\end{proof}

The multisecants constructed in Theorem \ref{thm:multi} should be quite special, the reason is twofold. First, in view of Proposition \ref{p1}, we do not expect that the Gauss map to be constant on the intersection of a high dimensional multisecant of the theta divisor and the smooth part of the theta divisor. Secondly, we do not know if the Gunning multisecant formula \ref{Gunn} describes all multisecants of a Jacobian. For this particular multisecants, we are to prove the following generalization of Theorem \ref{thm:proof}.

\begin{proposition}
Let $a_1,\ldots,a_\ell$ be the points from Theorem \ref{Gunn}, assume that they all lie on the smooth locus of $\Theta$ and on the same fiber of the Gauss map. Then $\G(a_i)\in\mathcal{B}_\ell$.
\end{proposition}
\begin{proof}
The proof is similar to the proof of Proposition \ref{centralprop}. Since each $a_i$ lies on the smooth locus of $\Theta$, then 
$$a_i=D_i-\kappa$$
for some effective divisor $D_i$ with $l(D_i)=1$. This implies that there is an effective canonical divisor $K_0$ such that
$$K_0=D_i+E_i$$
for all $i$ (note that this is equality of divisors, not linear equivalence). In particular, we see that
$$a_i-a_j\equiv p_i-p_j\equiv D_i-D_j\equiv E_j-E_i$$
for all $i,j$. By Lemma \ref{lem:trick}, we have that $p_i$ is in the support of $D_i$ and $E_j$ for all $i\neq j$. Now since $D_i-p_i$ is effective, $l(D_i-p_i)=1$, and
$$D_i-p_i\equiv D_j-p_j,$$
we have that $D_i+p_j=D_j+p_i$ for all $i,j$ (note again, this is equality of divisors). Therefore, there exist points $P_1,\ldots,P_{g-2}\in C$ such that
$$D_i=P_1+\cdots+P_{g-2}+p_i$$
$$D_j=P_1+\cdots+P_{g-2}+p_j$$
for all $i,j$. Note as well by the previous discussion that $E_j$ contains each $p_i$ for $i\neq j$, and so 
$$E_j=p_1+\cdots+p_{j-1}+p_{j+1}+\cdots+p_\ell+F_j$$
for some effective divisor $F_j$. Now we see that
$$a_i+a_j\equiv D_i+D_j-K_0=D_i-E_j.$$
On the other hand,
$$a_i+a_j=\sum_{k=1}^{\ell-2}q_k-\sum_{\stackrel{k=1}{k\neq i,j}}^{\ell}p_k,$$
and so
$$D_i\equiv \sum_{k=1}^{\ell-2}q_k+E_j-\sum_{\stackrel{k=1}{k\neq i,j}}^\ell p_k.$$
Now since $l(D_i)=1$ and the right hand side is an effective divisor, we have equality of divisors, and so
$$D_i=\sum_{k=1}^{\ell-2}q_k+p_i+F_j.$$
In other words, we see that $F_j$ is a divisor that is independent of $j$, so we will call it $F$. We conclude by writing
\begin{eqnarray}
\nonumber K_0=D_i+E_i&=&p_1+\cdots+p_{i-1}+p_{i+1}+\cdots+p_\ell+F+\sum_{k=1}^{\ell-2}q_k+F+p_i\\
\nonumber &=&\sum_{i=1}^{\ell}p_i+\sum_{i=1}^{\ell-2}q_i+2F\in\mathcal{B}_\ell.\end{eqnarray}

\end{proof}

\end{section}

\begin{section}{Remarks about the linear system $\Gamma_{00}$ and the Gauss map}\label{Gamma}

In this section, we shall consider a partial converse of the situation described in Proposition \ref{p1}: we analyze the linear span of points in the fiber of the Gauss map inside $\Gamma(A, 2\Theta)$.

Let $(A,\Theta)$ be a principally polarized abelian variety. The space $\Gamma_{00}$ is a distinguished subspace of $\Gamma(A, 2\Theta)$, it consists of sections whose vanishing order at the origin is at least $4$. If $(A, \Theta)$ is indecomponsable, the dimension of $\Gamma_{00}$ is equal to $2^g-g(g+1)/2-1$, cf \cite{GG}. A  basis for these  spaces is  described in  \cite{gsm}. We  want to relate the fiber of the Gauss map with  this space. It is a well known fact that if $x\in  {\rm Sing}(\Theta)$, then $s_x(z)\in \Gamma_{00}$.

\begin{lemma}
 Let $x_1,\ldots,x_r\in\Theta^{sm}$ be points such that $\Km(x_1),\ldots,\Km(x_r)$ are different and  $\mathcal{G}(x_1)=\ldots=\mathcal{G}(x_r)$,  then there exist  constants
$\lambda_2,\cdots, \lambda_r$ such that
 $$s_{x_1}(z)-\lambda_j  s_{x_j}(z)\in\Gamma_{00}.$$ 
\end{lemma}
\begin{proof}  By hypothesis, for every $j$ there exists a complex number $\gamma_j$  such that
$$\frac{\partial\theta}{\partial z_i}(x_1)=\gamma_j\frac{\partial\theta}{\partial z_i}(x_j )$$
for every $i$. Hence setting $\lambda_j=\gamma_j^2$ we produce sections of $\Gamma_{00}$

\end{proof}

We have the following corollary of the previous lemma.
\begin{corollary}
 Let $x_1,\ldots,x_r\in\Theta^{sm}$ be points such that $\Km(x_1),\ldots,\Km(x_r)$ are different but  $\mathcal{G}(x_1)=\ldots=\mathcal{G}(x_r)$; then, the  projective space  generated by these points in $\mathbb{P}^{2^g-1}$ has dimension at  most $2^g-g(g+1)/2-1$
\end{corollary}

Thus for any point $p\in \mathbb{P}^{g-1}$ in the image of the Gauss map, it makes sense to consider the space  $V_p$ spanned by the sections described above and ask if it is the full $\Gamma_{00}$.

In the Jacobian case we can do something more. To start with, we have the following lemma.

\begin{lemma}
 Let $x_1,x_2,x_3\in\Theta$, not all singular points , thus  they determine a trisecant if and only if 
$${\rm dim}(\Gamma_{00}\cap\Span\{s_{x_1}(z), s_{x_2}(z), s_{x_3}(z)\})=1$$
\end{lemma}

In general let $K_0=\sum_{i=1}^{2g-2}P_i$ be an effective canonical divisor; thus any effective divisor $D$  of degree $g-1$ whose support is contained  in $K_0$ produces a point $x_D\in\Theta$, and we have $$s_{x_D}(z)\in \Gamma_{00}\iff  l(D)\geq 2  $$
For the l $D$s  such  that  $l(D)=1$, we fix one  $D_1$ and then we have that  
$$t_D(z)=s_{x_{D_1}}(z)-\lambda_Ds_{x_D}(z)\in \Gamma_{00} $$
 
 We conclude that in the Jacobian case we can enlarge $V_p$ as follows:  for any point $p=\mathcal G(K_0)\in \mathbb{P}^{g-1}$ we set   
 $$W_p=\Span\{\dots t_{x_D}(z)\dots, s_{x_E}(z)\dots \}$$
 with $D, E$ effective  of degree $g-1$,  whose supports are   contained  in $K_0$ and  $l(D)=1,\, l(E)\geq 2$ Of course, $V_p\subset W_p$

\begin{que}  With the above notations, for which $p\in \mathbb{P}^{g-1}$ do we have either $\Gamma_{00}=  V_p$ or $\Gamma_{00}=  W_p$ ? 
\end{que}

We observe that in  some special cases both inclusions fail. The simplest case is already in genus 3: we can consider a smooth plane quartic with a point $P_0$ such that $K_0= 4P_0$, then $\mathcal G^{-1}(K_0)$ is a single point with multiplicity six.

We conclude  with a last remark. We shall write $\Theta_x$ for $\Theta+x$. For fixed $p\in\mathbb P^{g-1}$, let us consider in $|2\Theta|$ the space $Z_p$ generated by the divisors
$$ \Theta_{x} \cup  \Theta_{-x}$$  with $x$ in $\Theta^{sing}$ or in $\mathcal G^{-1}(p)$. 
Since $x$ is in $\Theta$, then $0$ is in $\Theta_{x} \cup  \Theta_{-x}$, so in the base locus of this linear system. If it is an isolated point of the base locus, then $(A, \Theta)$  is  not a Jacobian. Hence, we can paraphrase \cite[Conjecture 1]{BD} as follows

\begin{con} If $(A, \Theta)$ is not a Jacobian, for some $p\in\mathbb P^{g-1}$ the base locus of the linear system $Z_p$ is zero dimensional in a neighborhood of the origin.
\end{con}

\end{section}

\end{document}